\documentclass{article}%
\usepackage{amsfonts}
\usepackage{amsmath}
\usepackage{amssymb}
\usepackage{graphicx}%
\providecommand{\U}[1]{\protect\rule{.1in}{.1in}}
\newtheorem{theorem}{Theorem}

\newtheorem{corollary}[theorem]{Corollary}

\newtheorem{definition}[theorem]{Definition}

\newtheorem{lemma}[theorem]{Lemma}

\newtheorem{proposition}[theorem]{Proposition}
\newtheorem{remark}[theorem]{Remark}

\newenvironment{proof}[1][Proof]{\noindent\textbf{#1.} }{\ \rule{0.5em}{0.5em}}
\begin{document}

\title{About Signature-Change Metrics on Manifolds.}
\author{Javier Lafuente-L\'{o}pez\\Departamento de \ Algebra \\Geometr\'{\i}a y Topolog\'{\i}a.\\Facultad de Matem\'{a}ticas.\\Universidad Complutense de Madrid. }
\maketitle

\begin{center}%
\begin{tabular}
[c]{l}%
\textbf{Abstract. }We provide a one-parameter family of Lorentz-Riemann\\
signature-change models of metric manifolds. This family generalizes\\
the Kossowski%
%TCIMACRO{\U{b4}}%
%BeginExpansion
\'{}%
%EndExpansion
s signature type-changihg stablished in \cite{kossow1}.\\
Simple local expressions are sought around the hypersurface of change.
\end{tabular}

\end{center}

\section{Introduction.}

The initial idea of signature change is due to Hartle and Hawking
\cite{Hawking} which makes that it possible to have both Euclidean and
Lorentzian regions in quantum gravity to avoid the initial spacetime
singularity predicted by the standard model. Changing signature spaces, were
initially studied by Larsen \cite{Larsen} but Kossowski an Kriele in
\cite{kossow1} marks the beginning of the transverse type-change model (see
definition below) which has used in a considerable number of publications of a
physical or geometric nature.

Following Kossowski's paper \cite{kossow1}, we establish the following

\begin{definition}
Let $M$ a differentiable manifold $g$ a symmetric $(0,2)$ smooth tensor field
in $M$ wich fails to have maximal rank on $\Sigma\subset M$. Is said $g$ is a
transverse type-changing metric, if around each point $p\in\Sigma$ there
exists a coordinate system $(x_{a})$ centred in $p$ such that if
\ $g=g_{ab}dx^{a}dx^{b}$ is the local expression of $g$, the differential of
the function $\det\left(  g_{ab}\right)  $ at $p$, does not vanish. This
condition (which is obviously independent of the chosen coordinate system)
implies that the set $\Sigma$ of singular points is a hypersurface (called a
singular hypersurface), and the signature changes by one unit when passing
through $\Sigma$.

If, in addition, at each singular point $p$, the radical $Rad(g_{p})$ is not
contained in $T_{p}\Sigma$, and $M\backslash\Sigma$ only has Lorentz or
Riemann components, we call $\left(  M,g\right)  $ a $\Sigma$-space.
\end{definition}

\begin{remark}
[Notational convention]The dimension of the manifold $M$ will be denoted by
$m\geq2$, and henceforth:%
\[%
\begin{tabular}
[c]{l}%
The indices $a,b,c,...$range between $1$ and $m$\\
while $i,j,k,...$ range between $1$ and $m-1$\\
differentiable means $C^{\infty}$.
\end{tabular}
\ \ \ \ \ \ \ \ \ \ \
\]

\end{remark}

Paper \cite{kossow1} begins the study of $\Sigma$-spaces, with the idea of
determining the geodesic lines that cross $\Sigma$ transversely. While the
results presented here are correct, some proofs are incomplete. For example,
to prove that around any singular point, there exist (normal) coordinates
$\left(  x_{i},x_{m}\right)  $ where $g=\sum g_{ij}dx_{i}dx_{j}-x_{m}%
dx_{m}^{2}$, he assumes without proof that is differentiable the function%
\[
F\left(  \lambda,t\right)  =sgn\left(  t\right)  \left(  \int_{0}^{t}%
\sqrt{\left\vert x\right\vert }\psi\left(  \lambda,x\right)  dx\right)  ^{2/3}%
\]
defined in an open $\Lambda\times\left]  -\varepsilon,\varepsilon\right[  $ of
$\mathbb{R}^{m+1}$, where $\psi$ is a smooth function with $\psi\left(
\lambda,0\right)  >0$. The aim of this work is to prove a more general result
such as, is differentiable the function
\[
F\left(  \lambda,t\right)  =sgn\left(  t\right)  \left(  \int_{0}%
^{t}\left\vert x\right\vert ^{r}\psi\left(  \lambda,x\right)  dx\right)
^{\frac{1}{r+1}}\text{, for }r\neq-1
\]
(see proof in the Appendix). Using of this result we initiated an analogous
study for the spaces that we have called $\Sigma^{\alpha}$-spaces (where
$\alpha>0$) which has a similar definition as $\Sigma$-spaces $\left(
M,g\right)  $ but now the $\alpha$-transversality condition is that, around
each point $p\in\Sigma$ there exists a coordinate system $(x_{a})$ such that
if \ $g=g_{ab}dx^{a}dx^{b}$ is the local expression of $g$, the function
$sign\left(  \det\left(  g_{ab}\right)  \right)  \left\vert \det\left(
g_{ab}\right)  \right\vert ^{\alpha}$ extends differentiably, and their
differential does not vanish at $p$ (see definition in section \ref{Sigma},
for more details). The (called signature change $\alpha$-transversal) metric
$g$, must be continous in $M$, and differentiable in $M\backslash\Sigma$.

A $\Sigma^{\alpha}$-space $\left(  M,g\right)  $ we say normal if around any
singular point $p$ there exists coordinate system $\left(  x_{i},x_{m}\right)
$, where the metric is written as $g=\sum g_{ij}dx_{i}dx_{j}-sign\left(
x_{m}\right)  \left\vert x_{m}\right\vert ^{1/\alpha}dx_{m}^{2}$. The main
result of this paper is that $\left(  M,g\right)  $ is normal if and only if
around each singular point there exist a differentiable geodesic line
distribution crossing $\Sigma$ in the radical direction.

\section{Preliminaries.}

The following functions are defined:
\begin{equation}
\epsilon(t)=\left\{
\begin{array}
[c]{l}%
+1\text{ if }t>0\\
0\text{ if }t=0\\
-1\text{ if }t<0
\end{array}
\right.  \text{, and for }\alpha\neq0,\text{ }t^{\left[  \alpha\right]
}=\epsilon(t)\left\vert t\right\vert ^{\alpha} \label{signo}%
\end{equation}
which satisfy the properties:%
\[%
\begin{tabular}
[c]{l}%
$t^{\left[  1\right]  }=t$, $t^{\left[  0\right]  }=\epsilon$, $\epsilon
^{2}=1$\\
$\left(  t^{\left[  \alpha\right]  }\right)  ^{\left[  \beta\right]
}=t^{\left[  \alpha\beta\right]  }$\\
$t^{[\alpha+\beta]}=\epsilon(t)t^{\left[  \alpha\right]  }t^{\left[
\beta\right]  },$ $t^{\left[  -\alpha\right]  }=\frac{1}{t^{\left[
\alpha\right]  }}$\\
$\left(  st\right)  ^{\left[  \alpha\right]  }=s^{\left[  \alpha\right]
}t^{\left[  \alpha\right]  }$%
\end{tabular}
\ \
\]
$M$ denotes a differentiable manifold of dimension $m$, and $\Sigma$ is a
hypersurface of $M$.%
\[%
\begin{tabular}
[c]{l}%
Henceforth, differentiable means $C^{\infty}$\\
and unless explicitly stated otherwise,\\
all elements of $M$, fields, submanifolds, etc.\\
shall be assumed to be of class $C^{\infty}$.
\end{tabular}
\
\]

Finally, if $X$ is a differentiable vector field on $M$, we denote by
$X_{t}\left(  p\right)  $ the (local) flow of $X$ on $M$, such that the curve
$\gamma_{p}=X_{t}\left(  p\right)  $ represents the integral curve of $X$
through the point $p$, i.e., $\gamma_{p}^{\prime}\left(  t\right)  =X\left(
\gamma_{p}\left(  t\right)  \right)  $ and $\gamma_{p}^{\prime}\left(
0\right)  =p$. The one-manifold defined by $\gamma_{p}$ is called line of $X$.

Let $\Sigma$ be a hypersurface of $M$

\begin{definition}
\label{Ab_Simple}A simple open set of $\left(  M,\Sigma\right)  $ is a
connected open set $M_{0}$ of $M$, such that $\Sigma_{0}=M_{0}\cap\Sigma
\neq\emptyset$ is connected, and $M_{0}\backslash\Sigma_{0}$ has exactly two
connected components.
\end{definition}

\begin{remark}
Note that for each point $p\in\Sigma$, one can find a neighborhood $M_{0}$
that is a simple open set. We call it a simple neighborhood of $p$. Simple
neighborhoods of $p$ obviously constitute a basis of neighborhoods for $p$.
\end{remark}

\begin{definition}
\label{EqSimple}A \emph{simple} equation for $\Sigma$, around a point
$p\in\Sigma$, is a differentiable function $\sigma:M_{0}\rightarrow\mathbb{R}$
defined on a neighborhood $M_{0}$ of $p$ in $M$, such that
\[
\left\{
\begin{array}
[c]{l}%
\Sigma_{0}=M_{0}\cap\Sigma=\left\{  x\in M_{0}:\sigma\left(  x\right)
=0\right\} \\
\left.  d\sigma\right\vert _{x}\neq0,\text{ }\forall x\in\Sigma_{0}%
\end{array}
\right.
\]
we then write $\Sigma:_{s}\sigma=0.$
\end{definition}

\begin{proposition}
\label{htau}If $\Sigma:_{s}\sigma=0$ and $\varphi:M_{0}\rightarrow\mathbb{R}$
is another equation around $p$ for $\Sigma$ (i.e. $\Sigma:\varphi=0$), then
there exists a differentiable $h:M_{0}\rightarrow\mathbb{R}$ such that
$\varphi=h\sigma$; furthermore, $\Sigma:_{s}\varphi=0 \Leftrightarrow h$ is
non-zero at the point.
\end{proposition}

\begin{proof}
We need the following preliminary result:

Let $\mathbb{R}_{0}^{m}$ be an open set of $\mathbb{R}^{m}$ and
$I_{\varepsilon} = \left]  -\varepsilon, \varepsilon\right[  \subset
\mathbb{R}$.

If $f=f\left(  x,t\right)  :\mathbb{R}_{0}^{m}\times I_{\varepsilon}
\rightarrow\mathbb{R}$ is $C^{k}$ differentiable, $k\geq1$, and $f\left(  x, 0
\right)  = 0$, then there exists $g=g\left(  x,t\right)  :\mathbb{R}_{0}%
^{m}\times I_{\varepsilon} \rightarrow\mathbb{R}$ of class $C^{k-1}$ such that
$f=tg\left(  x,t\right)  $. Indeed, for a fixed $t$, let $f_{t}=f_{t}\left(
x,s\right)  = f\left(  x,st\right)  $, then:
\[
f\left(  x,t\right)  = \left[  f_{t} \right]  _{s=0}^{s=1} = \int_{0}^{1}
\frac{\partial f_{t}}{\partial s} ds = t \int_{0}^{1} \left.  \frac{\partial
f}{\partial t} \right\vert _{st} ds
\]
and therefore the function
\[
g\left(  t\right)  = \int_{0}^{1} \left.  \frac{\partial f}{\partial t}
\right\vert _{st} ds
\]
is a function of class $C^{k-1}$ that satisfies $f\left(  t\right)  =
tg\left(  t\right)  $.

To prove the proposition now, we can construct a chart $\left(  x_{i},x_{m}
\right)  $ in the open set $M_{0}$ with $x_{m}=\sigma$, and then the function
$\varphi=\varphi\left(  x_{i},x_{m} \right)  $ satisfies $\varphi\left(
x_{i},0 \right)  = 0$. Using the preliminary result, it follows that there
exists a differentiable $h=h\left(  x_{i},x_{m} \right)  : M_{0}
\rightarrow\mathbb{R}$ such that $\varphi=hx_{m}=h\sigma$. Furthermore:
\[
\Sigma:_{s}\varphi=0 \Longleftrightarrow0 \neq\left.  \frac{\partial\varphi
}{\partial x_{m}} \right\vert _{x_{m}=0} = h\left(  x_{i},0 \right)
\]

\end{proof}

\begin{remark}
\label{Remhtau}Note that if $\Sigma:_{s}\sigma=0$ is a simple equation, then
$e^{\varphi}\sigma=0$ is also a simple equation of $\Sigma$ for any
differentiable $\varphi:M_{0}\rightarrow\mathbb{R}$.
\end{remark}

\begin{remark}
In particular, if $\varphi=\varphi\left(  x_{i},x_{m} \right)  $ is a
differentiable function on an open set of $\mathbb{R}^{m}$ and $\varphi\left(
x_{i},0 \right)  = 0$, then there exists $\psi=\psi\left(  x_{i},x_{m}
\right)  $ such that $\varphi=x_{m}\psi$.
\end{remark}

\begin{definition}
\label{SD}A $\Sigma$-distribution around $p_{0}\in\Sigma$ is an integrable
distribution $\mathcal{H}$ defined on a simple neighborhood $M_{0}$ of $p_{0}%
$, such that $\mathcal{H}_{p}=T_{p}\Sigma$ for all $p\in\Sigma_{0}=\Sigma\cap
M_{0}.$
\end{definition}

\begin{theorem}
\label{xm0}Let $X$ be a field on $M$ with $X\left(  x\right)  \notin
T_{x}\Sigma$ for all $x\in\Sigma$. Then for each $p_{0}\in\Sigma$, a
neighborhood $M_{0}$ in $M$ and a chart $\left(  u_{i} \right)  $ on
$\Sigma_{0}=M_{0}\cap\Sigma$ can be taken. Taking $M_{0}$ sufficiently small,
a chart $\left(  x_{i},x_{m} \right)  $ can be constructed such that:
\[
X=\frac{\partial}{\partial x_{m}}\text{, }\Sigma:x_{m}=0,\text{ }\left.  x_{i}
\right\vert _{\Sigma_{0}}=u_{i}
\]

\end{theorem}

\begin{proof}
Given $p_{0}\in\Sigma$ and a chart $\left(  u_{i} \right)  $ in a neighborhood
$\Sigma_{0}$ of $p_{0}$ in $\Sigma$, one can take an open set $M_{0}$ of $M$
around $p_{0}$ and $\varepsilon>0$ such that the map
\[
\Psi:\Sigma_{0}\times\left(  -\varepsilon,\varepsilon\right)  \rightarrow
M_{0}
\]
is a diffeomorphism, where $\Psi=\Psi\left(  p,t\right)  = X_{t}\left(
p\right)  $ is the local flow of $X$. This is because $\left.  d\Psi
\right\vert _{\left(  0,p_{0}\right)  } : T_{p_{0}}\Sigma\times\mathbb{R}
\rightarrow T_{p_{0}}M$ is an isomorphism. Indeed, if $v\in T_{p_{0}}%
\Sigma_{0}$ and $\gamma_{p}=\gamma_{p}\left(  t\right)  = \Psi\left(
p,t\right)  $ is the integral curve of $X$ through $p$, taking a curve
$\alpha:I_{\varepsilon}\rightarrow\Sigma_{0}$ with $\alpha^{\prime}\left(
0\right)  = v$, we have:
\begin{align*}
\left.  d\Psi\right\vert _{\left(  0,p_{0} \right)  } \left(  v,0 \right)   &
= \left.  \frac{d}{dt} \right\vert _{t=0} \Psi\left(  \alpha\left(  t \right)
, 0 \right) \\
&  = \left.  \frac{d}{dt} \right\vert _{t=0} \gamma_{\alpha\left(  t \right)
} \left(  0 \right)  = \alpha^{\prime} \left(  0 \right)  = v
\end{align*}%
\begin{align*}
\left.  d\Psi\right\vert _{\left(  0,p_{0} \right)  } \left(  0,1 \right)   &
= \left.  \frac{d}{dt} \right\vert _{t=0} \Psi\left(  p_{0}, t \right) \\
&  = \left.  \frac{d}{dt} \right\vert _{t=0} \gamma_{p_{0}} \left(  t \right)
= X \left(  p_{0} \right)
\end{align*}

Then the chart $\left(  x_{i},x_{m} \right)  $ is constructed with
$x_{i}=u_{i}\circ\pi_{2}\circ\Psi^{-1}$ and $x_{m}=\pi_{1}\circ\Psi^{-1}$,
where $\pi_{1},\pi_{2}$ are the corresponding projections of $\left(
-\varepsilon, \varepsilon\right)  \times\Sigma_{0}$ onto its factors,
following the scheme:%

\[%
\begin{tabular}
[c]{ccccc}
&  & $\left(  -\varepsilon,\varepsilon\right)  $ &  & \\
& $\overset{x_{m}}{\nearrow}$ & $\pi_{1}\uparrow$ &  & \\
$M_{0}$ & $\overset{\Psi^{-1}}{\longrightarrow}$ & $\left(  -\varepsilon
,\varepsilon\right)  \times\Sigma_{0}$ &  & \\
& $\overset{\left(  x_{i}\right)  }{\searrow}$ & $\pi_{2}\downarrow$ &  & \\
&  & $\Sigma_{0}$ &  &
\end{tabular}
\ \ \ \ \ \ \
\]

\end{proof}

\section{$\Sigma^{\alpha}$-spaces.\label{Sigma}}

In what follows, $M$ is a differentiable manifold of dimension $m$.

$(M,g)$ is said to be a $\Sigma^{\alpha}$-space ($\alpha>0$) if $g$ is a
continuous and symmetric tensor field of type $(0,2)$ and the set $\Sigma$ of
singular points is a hypersurface $\Sigma\subset M$. The following properties
are also required

\begin{itemize}
\item $(\Sigma, g|_{\Sigma})$ is a Riemannian manifold.

\item $g$ is differentiable on $M \setminus\Sigma$.

\item \textbf{$\alpha$-transversality}: Around each point $p \in\Sigma$, there
exists a coordinate system $(x_{i}, x_{m})$ such that if $g = g_{ab}dx^{a}
dx^{b}$, then the equation $\det(g_{ab})^{[\alpha]} = 0$ defines by
extension\footnote{This means that the function $\det(g_{ab})^{[\alpha]}$
extends differentiably to $\Sigma$} a simple equation of $\Sigma$.
\end{itemize}

\begin{remark}
The property of $\alpha$-transversality is maintained for the matrix of $g$
with respect to any chart or even with respect to any parallelization
$(E_{i},E_{m})$ around a point in $\Sigma$. Since the metric $g|_{\Sigma}$ is
Riemannian, it follows that $\mathrm{Rad}(g_{p})\in T_{p}M\setminus
T_{p}\Sigma$ for all $p\in\Sigma$.
\end{remark}

\begin{definition}
Let $(M,g)$ be a $\Sigma^{\alpha}$. A differentiable field $\rho$ on $M$ is called:

\begin{itemize}
\item \textbf{radical}: If it has no zeros, and for all $p\in\Sigma$,
$\rho(p)\in\mathrm{Rad}(g_{p})$.

\item \textbf{radical geodesic: }If it is radical, and all its integral lines
are geodesic lines

\item \textbf{$\Sigma$-Adapted}: If the distribution $\rho^{\bot}$ defines (by
extension) a $\Sigma$-distribution.

\item \textbf{synchronized}: If it is $\Sigma$-Adapted and $\Sigma_{t}%
=\{\rho_{t}(p):p\in\Sigma\}$ is a leaf of $\rho^{\bot}$ for all $t$.
\end{itemize}
\end{definition}

\begin{definition}
A coordinate system $(x_{i},x_{m})$ in a simple neighborhood of a point
$p\in\Sigma$ where the matrix $(g_{ab})$ of the metric $g=g_{ab}dx_{a}dx_{b}$
is of the form:
\begin{equation}
=%
\begin{pmatrix}
g_{ij} & 0\\
0 & \hbar x_{m}^{[1/\alpha]}%
\end{pmatrix}
,\quad\text{$\hbar$ is differentiable and never null.}\label{ESP}%
\end{equation}
is called \textbf{special coordinates}. (Note that $\partial/\partial x_{m}$
is then radical and $\Sigma$-adapted vector field).

If $\hbar=1$, they are called \textbf{normal coordinates}.
\end{definition}

\begin{proposition}
\label{Pr_2} If for a differentiable function $\sigma:M\rightarrow\mathbb{R}$,
$\Sigma:_{s}\sigma=0$, the field $G_{\sigma}$
\begin{equation}
G^{\sigma}=\frac{\operatorname{grad}\sigma}{\langle\operatorname{grad}%
\sigma,\operatorname{grad}\sigma\rangle}\label{Gsigma}%
\end{equation}
extends differentiably to $M$, then in a simple neighborhood of each point
$p\in\Sigma$, special coordinates $(x_{i},x_{m})$ can be constructed with
$G_{\sigma}=\partial/\partial x_{m}$. In particular $G_{\sigma}$ is a radical
and $\Sigma$-Adapted vectorfield.
\end{proposition}

\begin{proof}
With $G_{\sigma}$, a chart $(x_{i},x_{m})$ can be built where $G_{\sigma
}=\partial/\partial x_{m}$. For $x_{m}\neq0$, the metric matrix verifies
$g_{im}=\langle G_{\sigma},\partial/\partial x_{i}\rangle_{g}=0$ since
$G_{\sigma}\bot\Sigma_{t}$. As $g_{im}$ are continuous, $g_{im}(x_{i},0)=0$.
By the $\alpha$-transversality condition, $g_{mm}$ is of the form $\hbar
x_{m}^{[1/\alpha]}$ for some differentiable, non-null $\hbar$.
\end{proof}

\begin{proposition}
\label{Pr_1} If for a differentiable function $\sigma:M\rightarrow\mathbb{R}$,
with $\Sigma:_{s}\sigma=0$, the field $G^{\sigma}$ extends differentiably to
$M$, then it is a synchronized radical field. Conversely, if $\xi$ is a
synchronized radical field, there exists $\sigma:M\rightarrow\mathbb{R}$ such
that $\xi=G_{\sigma}$.
\end{proposition}

\begin{proof}
If $G_{\sigma}$ extends differentiably $0$, let $G_{t}^{\sigma}(p)$ be the
flow of $G_{\sigma}$; let us see then that $G^{\sigma}$ is a synchronized
radical field. By Proposition \ref{Pr_2} $G^{\sigma}$ is a radical and
$\Sigma$-Adapted vectorfield; furthermore, for $p\in\Sigma$:
\begin{align*}
\frac{d}{dt}(\sigma(G_{t}^{\sigma}(p))) &  =d\sigma(\gamma_{p}^{\prime
}(t))=\langle\operatorname{grad}\sigma,\gamma_{p}^{\prime}(t)\rangle\\
&  =\left\langle \operatorname{grad}\sigma,\left.  \frac{\operatorname{grad}%
\sigma}{\langle\operatorname{grad}\sigma,\operatorname{grad}\sigma\rangle
}\right\vert _{\gamma_{p}(t)}\right\rangle =1
\end{align*}
Therefore, the map $t\mapsto\sigma(G_{t}^{\sigma}(p))=t+\text{const}$, and
since $G_{0}^{\sigma}(p)=p\in\Sigma$, we conclude that $\text{const}%
=\sigma(p)=0$, thus $\sigma(G_{t}^{\sigma}(p))=t$, and $\xi=G_{\sigma}$ is a
synchronized radical field.

If $\xi$ is a synchronized radical field, let us take $M_{0}$ and
$\varepsilon$ small enough so that the following map is a diffeomorphism:
\begin{equation}
\Psi:\Sigma_{0}\times I_{\varepsilon}\rightarrow M_{0},\quad(p,t)\mapsto
\xi_{t}(p)\label{nusigma}%
\end{equation}
Then locally we can see the distribution $\xi^{\bot}$ as $\{\Sigma_{t}%
:\sigma=t\}$ for the function $\sigma=\pi_{2}\circ\Psi^{-1}$, since for each
$p\in\Sigma$:
\begin{align*}
\sigma(\xi_{t}(p)) &  =(\pi_{2}\circ\Psi^{-1})(\xi_{t}(p))\\
&  =\pi_{2}(p,t)=t
\end{align*}
And so $\xi=G^{\sigma}$, since the curves $t\mapsto\xi_{t}(p)$ and $t\mapsto
G_{t}^{\sigma}(p)$ coincide as $\sigma(\xi_{t}(p))=\sigma(G_{t}^{\sigma
}(p))=t$, defining both integral lines of $\xi^{\bot}$ through each
$p\in\Sigma$.
\end{proof}

\begin{proposition}
\label{Pr_0} If $\rho$ is a $\Sigma$-Adapted field, there exists a
synchronized field $\xi=e^{\varphi}\rho$ (with the same integral lines as
$\rho$).
\end{proposition}

\begin{proof}
We take $\rho$ as a radical field in a simple open set $M_{0}$, with
$\rho^{\bot}=\{\Sigma_{t}:\sigma=t\}$, with flow $\rho_{s}(p)=\alpha_{p}(s)$
we can take $M_{0}$ small enough so that $\Phi:\Sigma_{0}\times I_{\varepsilon
}\rightarrow M_{0},(p,s)\mapsto\alpha_{p}(s)$ is a diffeomorphism (see Theorem
\ref{xm0}), and for each $p\in\Sigma_{0}$, let $\gamma_{p}=\gamma
_{p}(t)=\alpha_{p}(s_{p}(t))\in\Sigma_{t}$ where $s=s(p,t)=s_{p}(t)$ is the
reparameterization of $\alpha_{p}=\alpha_{p}(s)$ such that $\gamma_{p}%
(t)\in\Sigma_{t}$.

Taking $\varepsilon>0$ and reducing $M_{0}$, we can assume the map
$\Psi:\Sigma_{0}\times I_{\varepsilon}\rightarrow M_{0},(p,t)\mapsto\gamma
_{p}(t)$ is a diffeomorphism verifying $(\sigma\circ\Psi)|_{(p,t)}=t$ As the
curves $\gamma_{p}$ are reparameterizations of $\alpha_{p}$ (integral curves
of $\rho$), the field $\xi$ that has $\gamma_{p}(t)$ as integral curves has
the same integral lines as $\rho$Thus they are proportional ($\xi=e^{\varphi
}\rho$) and $\xi^{\bot}=\rho^{\bot}$, making $\xi$ the desired radical field.
\end{proof}

\begin{corollary}
\label{Cor1} A $\Sigma$-Adapted field is necessarily a radical field.
\end{corollary}

\begin{proof}
Let $\rho$ be a $\Sigma$-Adapted field. Around $\Sigma$ we can define
$\rho^{\bot}$ as $\{\Sigma_{t}:\sigma=t\}$, and by Proposition \ref{Pr_0},
take $\xi=e^{\varphi}\rho$ synchronized such that $\xi_{t}(p)\in\Sigma_{t}$.
Then by Proposition \ref{Pr_1}, necessarily $\xi=G_{\sigma}$ which is radical.
Therefore $\rho=e^{-\varphi}\xi$ is a radical field.
\end{proof}

\begin{proposition}
\label{radadap}A radical geodesic field is necessarily a $\Sigma$-Adapted field.
\end{proposition}

\begin{proof}
Given the radical geodesic field $\rho$, we will use its flow /as in Theorem
\ref{xm0}) to obtain a chart $(x_{i},x_{m})$ in a simple open set $M_{0}$ with
$\rho=\partial/\partial x_{m}$ and $\Sigma:_{s}x_{m}=0$. Let be%
\[
(g_{ab})=%
\begin{pmatrix}
g_{ij} & g_{im}\\
g_{im} & g_{mm}%
\end{pmatrix}
\]
the metric matrix in these coordinates.

first we work in $M_{0}^{+}:x_{m}>0$ and reparameterize each $\gamma_{p}$ by
arc length. We seek a parameter change $t=t(x_{i}\left(  p\right)
,\widetilde{t})$ such that if $\ \widetilde{\gamma_{p}}(\widetilde{t}%
)=\gamma_{p}(t(x_{i}\left(  p\right)  ,\widetilde{t}))$, then:
\[
g(\gamma_{p}^{\prime}(t),\gamma_{p}^{\prime}(t))=g\left(  \frac{\partial
}{\partial x_{m}},\frac{\partial}{\partial x_{m}}\right)  _{x_{m}=t}%
=g_{mm}(\gamma_{p}(t))
\]
We have:
\[
1=\left\vert \frac{d\widetilde{\gamma_{p}}}{d\widetilde{t}}\right\vert
=\sqrt{g_{mm}(\gamma_{p}(t))}\frac{dt}{d\widetilde{t}}\Rightarrow\widetilde
{t}=\int_{0}^{t}\sqrt{g_{mm}(\gamma_{p}(t))}dt
\]
We then, take coordinate:
\begin{equation}
\left\{
\begin{array}
[c]{l}%
\widetilde{x}_{i}=x_{i}\\
\widetilde{x}_{m}=\int_{0}^{t}\sqrt{g_{mm}\left(  x_{i},x_{m}\right)  }dx_{m}%
\end{array}
\right.  \label{xtilde}%
\end{equation}
In these coordinates, the matrix of $g$ is $%
\begin{pmatrix}
\widetilde{g}_{ij} & \widetilde{g}_{im}\\
\widetilde{g}_{im} & 1
\end{pmatrix}
$ Now the curves $\widetilde{\gamma}_{p}$ are geodesics. Then $\widetilde
{g}_{im}=0$ because:
\begin{align*}
\frac{\partial\widetilde{g}_{im}}{\partial\widetilde{x}^{m}} &  =\frac
{\partial}{\partial\widetilde{x}_{i}}\widetilde{g}\left(  \frac{\partial
}{\partial\widetilde{x}_{i}},\frac{\partial}{\partial\widetilde{x}_{m}%
}\right)  =g\left(  \widetilde{\nabla}_{\frac{\partial}{\partial\widetilde
{x}_{m}}}\frac{\partial}{\partial\widetilde{x}_{i}},\frac{\partial}%
{\partial\widetilde{x}_{m}}\right)  =g\left(  \widetilde{\nabla}%
_{\frac{\partial}{\partial\widetilde{x}_{i}}}\frac{\partial}{\partial
\widetilde{x}_{m}},\frac{\partial}{\partial\widetilde{x}_{m}}\right)  \\
&  =\frac{1}{2}\frac{\partial}{\partial\widetilde{x}_{m}}g\left(
\frac{\partial}{\partial\widetilde{x}_{m}},\frac{\partial}{\partial
\widetilde{x}_{m}}\right)  =\frac{1}{2}\frac{\partial}{\partial\widetilde
{x}_{m}}\left(  1\right)  =0
\end{align*}
Thus, the function $\widetilde{x}_{m}\mapsto\widetilde{g}_{im}\left(
x_{i},x_{m}\right)  $ is constant. Since the coordinate change \ref{xtilde} is
continuous at $x_{m}\geq0$ and $\widetilde{g}_{im}(\widetilde{x}_{i},0)=0$, we
conclude $\widetilde{g}_{im}=0$.

A similar reasoning applies to the zone $x_{m}<0$. Denote $\gamma_{p_{0}%
}(t)=\rho_{t}(p_{0})$ for $p_{0}\in\Sigma_{0}$, we define $\sigma
:M_{0}\rightarrow\mathbb{R}$ as the length application:
\[
\sigma(p)=%
\begin{cases}
\text{\textrm{len}}(\gamma_{p_{0}}|_{[0,x_{m}\left(  p\right)  ]}) & \text{if
}x_{m}\left(  p\right)  \geq0\\
-\text{\textrm{len}}(\gamma_{p_{0}}|_{[0,x_{m}\left(  p\right)  ]}) & \text{if
}x_{m}\left(  p\right)  <0
\end{cases}
\]
This is a continuous map where $\Sigma_{0}:\sigma=0$. The level surfaces
$\Sigma_{s}:\sigma=s$ form a continuous distribution of hyperplanes that
coincide for $s\neq0$ with $\rho^{\perp}$). Consequently, it is concluded that
$\Sigma_{0}\in\rho^{\perp}$, $\rho$ is $\Sigma$-Adapted.
\end{proof}

\begin{corollary}
\label{Simple}If $\rho$ is a $\Sigma$-Adapted field, then $\langle\rho
,\rho\rangle^{\lbrack\alpha]}=0$ is a simple equation for $\Sigma$.
\end{corollary}

\begin{proof}
By Proposition \ref{Pr_0} we construct synchronized $\xi=e^{\varphi}\rho$, and
by Proposition \ref{Pr_1} there exists $\sigma$ such that $\xi=G_{\sigma}$.
Using Proposition \ref{Pr_2}, we construct special coordinates $(x_{i},x_{m})$
where $\xi=\partial/\partial x_{m}$. By $\alpha$-transversality, $\det
(g_{ab})^{[\alpha]}=0$ is a simple equation for $\Sigma$. Since $\det(g_{ij})$
is positive, $\det(g_{ab})^{[\alpha]}=\langle\xi,\xi\rangle^{\lbrack\alpha
]}\det(g_{ij})^{\alpha}$Thus $\langle\xi,\xi\rangle^{\lbrack\alpha]}=0$ is a
simple equation for $\Sigma$, and so is $\langle\rho,\rho\rangle
^{\lbrack\alpha]}=0$.
\end{proof}

\begin{theorem}
[main]\label{Th_2} If $(x_{i},x_{m})$ are normal coordinates, then
$\partial/\partial x_{m}$ is a $\Sigma$-Adapted geodesic field Conversely, if
$\rho$ is a radical geodesic field, then around each singular point there
exist normal coordinates $(x_{i},x_{m})$ with $\rho=\hbar\partial/\partial
x_{m}$, where $\hbar$ is never null..
\end{theorem}

\begin{proof}
The first statement is evident. For the converse, we first observe that in a
normal chart the curves $\gamma_{p}$ defined by:
\[
\gamma_{p}:\left\{
\begin{tabular}
[c]{l}%
$x=x_{i}\left(  p\right)  $\\
$x_{m}=s$%
\end{tabular}
\ \ \ \right.
\]
for $p\in\Sigma$, satisfy
\[
\left.  \left\langle \frac{d\gamma_{p}}{ds},\frac{d\gamma_{p}}{ds}%
\right\rangle \right\vert _{s}=\left.  \left\langle \frac{\partial}{\partial
x_{m}},\frac{\partial}{\partial x_{m}}\right\rangle \right\vert _{x_{m}%
=s}=s^{[1/\alpha]}%
\]
because of this, they are said to be $\alpha$-parameterized.

Given the radical geodesic field $\rho$, we will use its flow (see Theorem
\ref{xm0}) to obtain a chart $(x_{i},x_{m})$ in a simple open set $M_{0}$ with
$\rho=\partial/\partial x_{m}$ and $\Sigma:_{s}x_{m}=0$. Let be%
\[
(g_{ab})=%
\begin{pmatrix}
g_{ij} & g_{im}\\
g_{im} & g_{mm}%
\end{pmatrix}
\]
the metric matrix in these coordinates. By Proposition \ref{radadap} $\rho$ is
$\Sigma$-Adapted, and by corollary \ref{Simple} $\langle\rho,\rho
\rangle^{\lbrack\alpha]}=0$ is a simple equation for $\Sigma$.

We are going to $\alpha$-parameterize the curves $\gamma_{p}=\rho_{t}(p)$,
which are defined by $\gamma_{p}:x=x_{i}(p)$, $x_{m}=t$.

If $\gamma(t)$ is an integral curve of $\rho$, with $\gamma\left(  0\right)
\in\Sigma$, then since $\Sigma:_{s}\langle\rho,\rho\rangle^{\lbrack\alpha]}%
=0$, the function $\varphi=\langle\frac{d\gamma}{dt},\frac{d\gamma}{dt}%
\rangle^{\lbrack\alpha]}$ defines a simple equation $\varphi(t)=0$ of
$\{0\}\subset\mathbb{R}$ Thus $\varphi(t)=t\psi(t)^{[2\alpha]}$ for a
differentiable function $\psi$ with $\psi(0)>0$.
\[
\varphi\left(  t\right)  =t\psi\left(  t\right)  ^{\left[  2\alpha\right]
}\Rightarrow\left\langle \frac{d\gamma}{dt},\frac{d\gamma}{dt}\right\rangle
=\left[  t^{\frac{1}{2\alpha}}\psi\left(  t\right)  \right]  ^{2}%
\]

Baldomero's Theorem allows us to $\alpha$-parameterize $\gamma(t)$. We seek
$t=t(s)$ such that if $\overline{\gamma}\left(  s\right)  =\gamma\left(
t\left(  s\right)  \right)  $ we have
\begin{align*}
s^{\left[  1/\alpha\right]  } &  =\left\langle \frac{d\overline{\gamma}}%
{ds},\frac{d\overline{\gamma}}{ds}\right\rangle =\left(  \frac{dt}{ds}\right)
^{2}\left\langle \frac{d\gamma}{dt},\frac{d\gamma}{dt}\right\rangle \\
&  =\left(  \frac{dt}{ds}\right)  ^{2}\left[  t^{\left[  \frac{1}{2\alpha
}\right]  }\psi\left(  t\right)  \right]  ^{2}%
\end{align*}

therefore
\[
s^{\left[  \frac{1}{2\alpha}\right]  }=\frac{dt}{ds}t^{\left[  \frac
{1}{2\alpha}\right]  }\psi\left(  t\right)  \Rightarrow\frac{s^{1+\frac
{1}{2\alpha}}}{\frac{1}{2\alpha}+1}=\epsilon\int_{0}^{t}\tau^{\left[  \frac
{1}{2\alpha}\right]  }\psi\left(  \tau\right)  d\tau
\]%
\[
s\left(  t\right)  =\left(  \left(  1+\frac{1}{2\alpha}\right)  \int_{0}%
^{t}\left(  \epsilon\tau\right)  ^{\left[  \frac{1}{2\alpha}\right]  }%
\psi\left(  \tau\right)  d\tau\right)  ^{\left[  1+\frac{1}{2\alpha}\right]  }%
\]
which it is smooth by Baldomero%
%TCIMACRO{\U{b4}}%
%BeginExpansion
\'{}%
%EndExpansion
s theorem for $r=\frac{1}{2\alpha}.$

In general taking $\gamma=\gamma_{p}:\left\{
\begin{tabular}
[c]{l}%
$x_{i}=x_{i}\left(  p\right)  $\\
$x_{m}=t$%
\end{tabular}
\ \ \ \ \ \ \right.  $ $x_{m}\left(  p\right)  =0$, we make the functions
$\psi\left(  x_{i},t\right)  $ such that%
\[
\left\langle \frac{d\gamma_{p}}{dt},\frac{d\gamma_{p}}{dt}\right\rangle
=\left[  t^{\left[  \frac{1}{2\alpha}\right]  }\psi\left(  x_{i}\left(
p\right)  ,t\right)  \right]  ^{2}%
\]
and the smooth function
\[
s\left(  x_{i},t\right)  =\left(  \left(  1+\frac{1}{2\alpha}\right)  \int
_{0}^{t}\left(  \epsilon\tau\right)  ^{\left[  \frac{1}{2\alpha}\right]  }%
\psi\left(  x_{i},\tau\right)  d\tau\right)  ^{\left[  1+\frac{1}{2\alpha
}\right]  }%
\]

defines the reparametrization $s=s\left(  x_{i}\left(  p\right)  ,t\right)  $
that $\alpha$-parameterizes $\gamma_{p}$, that is, if $\overline{\gamma}%
_{p}=\overline{\gamma}_{p}\left(  s\right)  $ is the reparametrized one such
that
\[
s^{\left[  1/\alpha\right]  }=\left\langle \frac{d\overline{\gamma}_{p}}%
{ds},\frac{d\overline{\gamma}_{p}}{ds}\right\rangle
\]
and we can write $\overline{\gamma}_{p}=\overline{\gamma}_{p}\left(  s\right)
:\left\{
\begin{tabular}
[c]{l}%
$\overline{x}_{i}=\overline{x}_{i}\left(  p\right)  $\\
$\overline{x}_{m}=s$%
\end{tabular}
\ \ \ \ \ \ \ \right.  $ in the coordinates given by%
\[
\left\{
\begin{tabular}
[c]{l}%
$\overline{x}_{i}=x_{i}$\\
$\overline{x}_{m}=\left(  \left(  1+\frac{1}{2\alpha}\right)  \int_{0}^{x_{m}%
}\left(  \epsilon\tau\right)  ^{\left[  \frac{1}{2\alpha}\right]  }\psi\left(
x_{i},\tau\right)  d\tau\right)  ^{\left[  1+\frac{1}{2\alpha}\right]  }$%
\end{tabular}
\ \ \ \ \ \right.
\]

At these coordinates $\left(  \overline{x}_{i},\overline{x}_{m}\right)  $ the
metric matrix is
\[
\left(  \overline{g}_{ab}\right)  =\left(
\begin{array}
[c]{cc}%
\overline{g}_{ij} & \overline{g}_{im}\\
\overline{g}_{im} & \overline{x}_{m}^{\left[  1/\alpha\right]  }%
\end{array}
\right)
\]
To prove that $\overline{g}_{im}$ is zero , we will proceed as in
(\ref{xtilde}). We define the continuous change of coordinates, $t=t\left(
s\right)  $ differentiable at $x_{m}>0$ that makes geodesics $\widetilde
{\gamma}_{p}=\overline{\gamma}_{p}\left(  t\left(  s\right)  \right)  $ to the
curves $\overline{\gamma}_{p}$. We have for $\alpha\neq1/2$
\begin{align*}
1 &  =\left\langle \frac{d\widetilde{\gamma}_{p}}{ds},\frac{d\widetilde
{\gamma}_{p}}{ds}\right\rangle =\left(  \frac{ds}{dt}\right)  ^{2}t^{\left[
1/\alpha\right]  }\Rightarrow\\
ds &  =t^{\left[  -1/\left(  2\alpha\right)  \right]  }dt\Rightarrow
s=\epsilon\left(  t\right)  \left\vert t\right\vert ^{\left(  2\alpha
-1\right)  /2\alpha}%
\end{align*}

This means that in the coordinates
\[
\left\{
\begin{array}
[c]{l}%
\widetilde{x}_{i}=\overline{x}_{i}\\
\widetilde{x}_{m}=\epsilon\left(  \overline{x}_{m}\right)  \left\vert
\overline{x}_{m}\right\vert ^{\left(  2\alpha-1\right)  /2\alpha}%
\end{array}
\right.
\]
the metric matrix is
\[
\left(  \widetilde{g}_{ab}\right)  =\left(
\begin{array}
[c]{cc}%
\widetilde{g}_{ij} & 0\\
0 & 1
\end{array}
\right)
\]

But now it turns out that%
\[
\frac{\partial}{\partial\overline{x}_{i}}=\frac{\partial}{\partial
\widetilde{x}_{i}},\frac{\partial}{\partial\overline{x}_{m}}=\frac
{\partial\widetilde{x}_{m}}{\partial\overline{x}_{m}}\frac{\partial}%
{\partial\widetilde{x}_{m}}=\overline{x}_{m}^{^{\left[  -1/2\alpha\right]  }%
}\frac{\partial}{\partial\widetilde{x}_{m}}%
\]
therefore
\[
\overline{g}_{im}=\left\langle \frac{\partial}{\partial\overline{x}_{i}}%
,\frac{\partial}{\partial\overline{x}_{m}}\right\rangle =\overline{x}%
_{m}^{^{\left[  -1/2\alpha\right]  }}\left\langle \frac{\partial}%
{\partial\widetilde{x}_{i}},\frac{\partial}{\partial\widetilde{x}_{m}%
}\right\rangle =0
\]

For $\alpha=1/2$ the change of coordinates would be%
\[
\left\{
\begin{array}
[c]{l}%
\widetilde{x}_{i}=\overline{x}_{i}\\
\widetilde{x}_{m}=\epsilon\left(  \overline{x}_{m}\right)  \ln\left\vert
\overline{x}_{m}\right\vert
\end{array}
\right.
\]
which unfortunately is not a continuous change of coordinates, but
nevertheless at $x_{m}\neq0$ it still happens that%
\[
\frac{\partial}{\partial\overline{x}_{i}}=\frac{\partial}{\partial
\widetilde{x}_{i}},\frac{\partial}{\partial\overline{x}_{m}}=\frac
{\partial\widetilde{x}_{m}}{\partial\overline{x}_{m}}\frac{\partial}%
{\partial\widetilde{x}_{m}}=\frac{\epsilon\left(  \overline{x}_{m}\right)
}{\overline{x}_{m}}\frac{\partial}{\partial\widetilde{x}_{m}}%
\]
therefore
\[
\overline{g}_{im}=\left\langle \frac{\partial}{\partial\overline{x}_{i}}%
,\frac{\partial}{\partial\overline{x}_{m}}\right\rangle =\frac{\epsilon\left(
\overline{x}_{m}\right)  }{\overline{x}_{m}}\left\langle \frac{\partial
}{\partial\widetilde{x}_{i}},\frac{\partial}{\partial\widetilde{x}_{m}%
}\right\rangle =0
\]
and by continuity is also $\overline{g}_{im}\left(  x_{i},0\right)  =0.$
\end{proof}

\section{Conclussion.}

We have constructed a one-parameter family, called the $\Sigma^{\alpha}%
$-spaces, which are models of metric manifolds $(M,g)$ such that they change
the signature from Lorentz to Riemann, when traversing a hypersurface $\Sigma$
called a singular. These spaces are of geometric interest in themselves but
they also have applications in cosmology.

Our $\Sigma^{\alpha}$-spaces spaces ($\alpha>0$) are characterized by inducing
a Riemannian metric on $\Sigma$, and that around each point $p\in\Sigma$,
there exists a local expression of $g=\Sigma g_{ab}dx_{a}dx_{b}$ such that the
function $sign\left(  \det\left(  g_{ab}\right)  \right)  \left\vert
\det\left(  g_{ab}\right)  \right\vert ^{\alpha}$ extends differently, and
their differential does not vanish at $p$. When $\alpha=1$ and $g$ is smooth,
we obtain the type-change model proposed by Kossowski in the foundational
paper \cite{kossow1} where it is shown that these metrics admit a local
representation of the form $g=\Sigma g_{ij}dx_{i}dx_{j}-x_{m}dx_{m}^{2}$
around each singular point. But it is used without proof, that for $r=1/2$ the
function $F\left(  \lambda,t\right)  =sgn\left(  t\right)  \left(  \int
_{0}^{t}\left\vert x\right\vert ^{r}\psi\left(  \lambda,x\right)  dx\right)
^{\frac{1}{r+1}}$ is differentiable and this is a far from trivial fact. We
have proven it for all $r>-1$, and this has motivated the definition of the
family of $\Sigma^{\alpha}$-spaces, and has allowed us to prove that in the
presence of a geodesic radical line distribution, there is a local expression
$g=\Sigma g_{ij}dx_{i}dx_{j}-sign\left(  x_{m}\right)  \left\vert
x_{m}\right\vert ^{1/\alpha}dx_{m}^{2}$ around each singular point.

\section{Discussion.}

\begin{enumerate}
\item The metric of a $\Sigma^{\alpha}$-space $\left(  M,g\right)  $ does not
have to be differentiable on all of $M$. For example, on$M=\mathbb{R}^{2}$
with coordinates $\left(  x,y\right)  $ and metric matrix
\[
\left(  g_{ab}\right)  =\left(
\begin{array}
[c]{cc}%
1 & y^{\left[  1/2\right]  }\\
y^{\left[  1/2\right]  } & 2y
\end{array}
\right)
\]
$\det\left(  g_{ab}\right)  =y$ and it is a $\Sigma^{1}$-space, but obviously
the differentiability of the metric fails at each point of the singular line
$\Sigma:y=0$. Therefore, this $\Sigma^{1}$--space is not of the Kossowski type
\cite{kossow1}. In fact, it cannot admit a distribution of radical geodesic
lines, because if it did, by our main theorem \ref{Th_2} there would be a
coordinate system $\left(  x_{1},x_{2}\right)  $ where the metric have a
matrix as $\left(
\begin{array}
[c]{cc}%
1 & 0\\
0 & x_{2}%
\end{array}
\right)  $, and would be differentiable at the points of $\Sigma:x_{2}=0$.
This proves that not every $\Sigma^{\alpha}$-espacio admits a distribution of
radical geodesic lines.

\item The metric $g$ on $\mathbb{R}^{m}=\left\{  \left(  x_{i},x_{m}\right)
\right\}  $ with matrix%
\begin{equation}
\left(  g_{ab}\right)  =\left(
\begin{array}
[c]{cc}%
g_{ij} & 0\\
0 & \hbar x_{m}^{\left[  1/\alpha\right]  }%
\end{array}
\right)  \text{, }%
\begin{tabular}
[c]{l}%
$\hbar$ smooth\\
and nonnull.
\end{tabular}
\ \label{esp}%
\end{equation}
define a $\Sigma^{\alpha}$-space, and the field $\rho=\partial/\partial x_{m}$
turns out to be a $\Sigma$-radical field. It can be shown that conversely, if
$\left(  M,g\right)  $ is a $\Sigma^{\alpha}$-space that admits a $\Sigma
$-radical field, it also admits local representations of the metric as in
(\ref{esp}) around $\Sigma$. We ask whether, in this type of $\Sigma^{\alpha}%
$-spaces, the existence of a distribution of radical geodesic lines around the
singular hypersurface can be guaranteed, and we might ask whether this
distribution is essentially unique.

\item There is a natural definition of $\Sigma^{\alpha}$-space for $\alpha<0$,
$\alpha\neq-1/2$; it consists of Lorentz or Riemann metrics $g$ defined on the
components of the complement $M\backslash\Sigma$ of the hypersurface $\Sigma$
such that their dual metric $g^{\ast}$ extends continuously over all of $M$
and satisfies a property of $\alpha$-transversality on every point of $\Sigma
$. Note that in this case the metric $g$ is not defined on $\Sigma$ now its
points become poles of the metric and $\Sigma$ is called polar hypersurface of
$\left(  M,g\right)  $. This is, in a sense, a dual situation to the case
$\alpha>0$; however, interpreting and dualizing the results is far from trivial.

See \cite{Laf} for a study of this topic for the case $\alpha=-1$.\newpage
\end{enumerate}

\section{APPENDIX: Baldomero's Theorem.}

Recall the definition of the sign function:
\[
\epsilon(t)=\left\{
\begin{array}
[c]{l}%
+1\text{ if }t>0\\
0\text{ if }t=0\\
-1\text{ if }t<0
\end{array}
\right.
\]

Fixed a real number $\alpha$ with $r\neq-1$, it is verified that%
\[
\int_{0}^{t}x^{r}dx=\frac{t^{r+1}}{r+1}%
\]
and therefore the function $F:\left]  0,\infty\right)  \rightarrow\mathbb{R}$
defined by
\[
F\left(  t\right)  =\left(  \int_{0}^{t}x^{r}dx\right)  ^{\frac{1}{r+1}}%
=\frac{t}{r+1}%
\]
is $C^{\infty}$ differentiable on $\left[  0,\infty\right)  $ by taking
$F\left(  0\right)  =0$.

Furthermore, for $t<0$, we have%
\begin{align*}
\int_{0}^{t}\left(  -x\right)  ^{r}dx  &  =\left[  -\frac{\left(  -x\right)
^{r+1}}{r+1}\right]  _{0}^{t}\\
&  =-\frac{\left(  -t\right)  ^{r+1}}{r+1}%
\end{align*}
hence for $t<0$
\[
\left(  \epsilon\left(  t\right)  \int_{0}^{t}\left\vert x\right\vert
^{r}dx\right)  ^{\frac{1}{r+1}}=\frac{-t}{r+1}%
\]
so the function $F:\mathbb{R}\rightarrow\mathbb{R}$ defined by%
\[
F\left(  t\right)  =\epsilon\left(  t\right)  \left(  \epsilon\left(
t\right)  \int_{0}^{t}\left\vert x\right\vert ^{r}dx\right)  ^{\frac{1}{r+1}%
}=\frac{t}{r+1}%
\]
is $C^{\infty}$ differentiable. Inspired by this, we conjecture the validity
of the following statement:

\begin{theorem}
[Baldomero]\label{Baldo gen} Assume that $\psi=\psi\left(  \lambda,t\right)
:\Lambda\times\mathcal{I}\rightarrow\mathbb{R}$ is a $C^{q}$ function
($q=1,2,...\infty$) defined on an open interval $\mathcal{I}$ of $\mathbb{R}$,
with $0\in\mathcal{I}$ and $\psi\left(  0\right)  >0$. Fixed a real number $r$
with $r\neq-1$, the function $F:\mathcal{I}\rightarrow\mathbb{R}$ defined by
\begin{equation}
F:t\mapsto F\left(  \lambda,t\right)  =\epsilon\left(  t\right)  \left\vert
\int_{0}^{t}\left\vert x\right\vert ^{r}\psi\left(  \lambda,x\right)
dx\right\vert ^{\frac{1}{r+1}} \label{Ft}%
\end{equation}
is a $C^{q}$ differentiable function, and if $q\geq1$, then $F^{\prime}\left(
0\right)  >0$.
\end{theorem}

\subsection{Preliminaries.}

The proof requires some preparations.

\subsubsection{Notations.\label{Notacion}}

If $\Omega$ is an open set in $\mathbb{R}^{n}$, $C^{q}\left(  \Omega\right)  $
denotes the set of functions $f:\Omega\rightarrow\mathbb{R}$ that are
differentiable of class $C^{q}$, $q=0,1,2,...,\infty$. We adopt the following
simplifications for functions $f:\Lambda\times\mathcal{I}\rightarrow
\mathbb{R}$:
\begin{align}
f  &  \in C^{q}\Leftrightarrow\exists\varepsilon\text{ such that }f\in
C^{q}\left(  \Lambda\times I_{\varepsilon}\right) \label{simplif}\\
f  &  \in C_{\ast}^{q}\Leftrightarrow\exists\varepsilon\text{ such that }f\in
C^{q}\left(  \Lambda\times I_{\varepsilon}^{\ast}\right) \nonumber\\
f  &  \in C_{\ast}^{\infty}\Leftrightarrow f\in C_{\ast}^{q}\text{, }\forall
q\geq0
\end{align}

\begin{definition}
We say that $f\in\mathcal{C}^{\omega}$ if $f\in C_{\ast}^{\infty}$ and if
there exists $\overline{f}\in C^{0}$ with $\overline{f}\left(  \lambda
,t\right)  =f\left(  \lambda,t\right)  $ if $t\neq0$.
\begin{equation}%
\begin{tabular}
[c]{l}%
We agree to denote $f=\overline{f}\in C^{0}$%
\end{tabular}
\label{Abuso}%
\end{equation}

\end{definition}

\begin{remark}
Using convention (\ref{Abuso}) and when $f\in C_{\ast}^{\infty}$, we can write%
\[
f\in C^{q}\Leftrightarrow f\in C^{\omega}\cap C^{q}%
\]%
\[%
\begin{tabular}
[c]{l}%
For $f\in C_{\ast}^{\infty}$ we shall write\\
$\lim f$ instead of $\lim_{t\rightarrow0}f$\\
and we denote $f^{\left(  q\right)  }=\frac{\partial^{q}f}{\partial t^{q}}\in
C_{\ast}^{\infty}$%
\end{tabular}
\
\]

\end{remark}

Furthermore, we denote $I=I\left(  \lambda,t\right)  =t$ as the identity map
and $J=J\left(  \lambda,t\right)  =\left\vert t\right\vert $ as the absolute
value map. Thus, $\epsilon\in C_{\ast}^{\infty}$ is the sign function:
\[
\epsilon\left(  \lambda,t\right)  =\left\{
\begin{tabular}
[c]{l}%
$1$ if $t>0$\\
$-1$ if $t<0$%
\end{tabular}
\ \right.
\]
it then follows that%
\[
J=\epsilon I\in C_{\ast}^{\infty}\text{ }%
\]
Observe that%
\begin{equation}
\left(  J^{q}\right)  ^{\prime}=\epsilon qJ^{q-1} \label{Jp}%
\end{equation}

If $f\in C_{\ast}^{\infty}$, we denote $\int f$ $\in C_{\ast}^{\infty}$ as the
function:%
\[
\int f:t\longmapsto\int_{0}^{t}f\left(  x\right)  dx
\]
Note that by the Fundamental Theorem of Calculus (Barrow's rule):
\[
\left(  \int f\right)  ^{\prime}=f
\]

With these notations, function \ref{Ft} can be written as%
\begin{equation}
F=\left(  \epsilon\int\psi J^{r}\right)  ^{\frac{1}{r+1}} \label{alpha}%
\end{equation}

Finally, to avoid overloading the calculations, we define the following
relation in $C_{\ast}^{\infty}$:

If $f,g\in C_{\ast}^{\infty}$, we write $f\cong g$ if and only if the
following condition holds:%
\begin{equation}
f\cong g\Leftrightarrow\exists\varphi\in C^{\infty}\text{ and }\exists
a\neq0\text{ such that }g=af+\varphi\label{==}%
\end{equation}
If $f,g\in C_{\ast}^{\infty}$ we write $f\simeq g$ if $f$ and $g$ are
functionally related such that the following equivalence holds:
\[
f\in\mathcal{C}^{q}\Leftrightarrow g\in\mathcal{C}^{q},\forall
q=0,1,2,...,\infty
\]

It is easy to see that both are equivalence relations, and for our purposes,
the relevant fact is that:
\begin{equation}
\text{ }f\cong g\Rightarrow f\simeq g\text{ }\nonumber
\end{equation}

However, the converse is not true; for example, if $\varphi\in C^{\infty}, $
the functions $f$ and $g=\left(  f+\varphi\right)  ^{3}$ satisfy $f\simeq g$
and yet $f\ncong g$. Note that if $g=$ $\varphi f$ with $\varphi\in C^{\infty
}$ and if $\varphi\left(  0\right)  \neq0$, then $f\cong g$. The following
result is essential:

\begin{lemma}
\label{lemain} If $f:I_{\varepsilon}\rightarrow\mathbb{R}$ is a continuous
function, of class $C^{1}$ in $I_{\varepsilon}^{\ast}=\left]  -\varepsilon
,\varepsilon\right[  \backslash\left\{  0\right\}  $ and there exists
$\ell=\lim_{t\rightarrow0}f^{\prime}(t)$, then
\[
f\left(  t\right)  -f\left(  0\right)  = t\int_{0}^{1}f^{\prime}\left(
st\right)  ds
\]
and in particular $f$ is of class $C^{1}$ in $I_{\varepsilon}$, and
$f^{\prime}(0)=\ell$.
\end{lemma}

\begin{proof}
For a fixed $t$, let $f_{t}=f_{t}\left(  s\right)  =f\left(  st\right)  $.
Then:
\[
f\left(  t\right)  -f\left(  0\right)  =\left[  f_{t}\right]  _{s=0}%
^{s=1}=\int_{0}^{1}\frac{df_{t}}{ds}ds=t\int_{0}^{1}f^{\prime}\left(
st\right)  ds
\]
therefore:
\begin{align*}
f^{\prime}\left(  0\right)   &  =\lim_{t\rightarrow0}\frac{f\left(  t\right)
-f\left(  0\right)  }{t}=\lim_{t\rightarrow0}\int_{0}^{1}f^{\prime}\left(
st\right)  ds=\\
&  =\int_{0}^{1}\lim_{t\rightarrow0}f^{\prime}\left(  st\right)  ds=\int
_{0}^{1}\ell ds=\ell
\end{align*}

\end{proof}

\subsubsection{Calculation Rules.}

\paragraph{Rule C1}%

\begin{equation}%
\begin{tabular}
[c]{l}%
If $m>1$, then $J^{m}\in C^{1}$ and $\left(  J^{m}\right)  ^{\prime} =
\epsilon mJ^{m-1}$%
\end{tabular}
\label{RC2}%
\end{equation}

\paragraph{Rule H.}

L'Hopital rule states:

If $f,g\in C^{1}\left(  I_{\varepsilon}^{\ast}\right)  $ satisfy $g^{\prime
}\left(  t\right)  \neq0$ $\forall t\neq0$, and $\lim_{t\rightarrow0}f\left(
t\right)  = \lim_{t\rightarrow0}g\left(  t\right)  = 0$, then:
\[
\exists\lim_{t\rightarrow0}\frac{f^{\prime}\left(  t\right)  }{g^{\prime
}\left(  t\right)  } = \ell\Rightarrow\lim_{t\rightarrow0}\frac{f\left(
t\right)  }{g\left(  t\right)  } = \ell
\]

Then for $\varphi\in\mathcal{C}^{0},m>1$:
\begin{equation}
\lim_{t\rightarrow0}\frac{\epsilon\int\varphi J^{n}}{J^{m}}=\lim
_{t\rightarrow0}\frac{\epsilon\varphi J^{n}}{\epsilon mJ^{m-1}}=\lim
_{t\rightarrow0}\varphi J^{n-m+1}=\left\{
\begin{array}
[c]{l}%
0\text{, if }n>m-1\\
\frac{1}{m}\varphi\left(  \lambda,0\right)  \text{ if }n=m-1
\end{array}
\right.  \label{RH}%
\end{equation}

and for $\varphi\in C^{\infty}$, and $n\geq m-1$:
\begin{equation}
\varphi_{n/m}=\frac{\epsilon\int\varphi J^{n}}{J^{m}}\in\mathcal{C}^{\omega
}\text{ and }\varphi_{n/m}\left(  \lambda,0\right)  =\left\{
\begin{array}
[c]{l}%
0\text{, if }n>m-1\\
\frac{1}{m}\varphi\left(  \lambda,0\right)  \text{ if }n=m-1
\end{array}
\right.  \label{RH1}%
\end{equation}
moreover, it holds that: $f\in C^{\omega}$, $f\left(  0\right)  =0\Rightarrow
\epsilon f\in C^{\omega}$, in particular:
\begin{equation}
\epsilon\varphi_{n/m}=\frac{\int\varphi J^{n}}{J^{m}}\in C^{\omega}\text{ if
}n>m-1 \label{RH2}%
\end{equation}

\begin{proof}
It is trivial using observation \ref{RH} to verify that for a fixed $\lambda=
\lambda_{0}$:
\[
\exists\lim_{t\rightarrow0}\varphi_{n/m}\left(  \lambda_{0},t\right)  =
\left\{
\begin{array}
[c]{l}%
0\text{, if } n>m-1\\
\frac{1}{m}\varphi\left(  \lambda_{0},0\right)  \text{ if } n=m-1
\end{array}
\right.
\]
but we need to prove that:
\[
\exists\lim_{\left(  \lambda,t\right)  \rightarrow\left(  \lambda
_{0},0\right)  }\varphi_{n/m}\left(  \lambda,t\right)  = \left\{
\begin{array}
[c]{l}%
0\text{, if } n>m-1\\
\frac{1}{m}\varphi\left(  \lambda_{0},0\right)  \text{ if } n=m-1
\end{array}
\right.
\]

Let us fix a small neighborhood $I_{\delta}\times\Lambda_{\delta}$ of $\left(
\lambda_{0},0\right)  $ and let:
\[
M = \sup_{\left(  \lambda,t\right)  \in I_{\delta}\times\Lambda_{\delta} }
\left\vert \varphi\left(  \lambda,t\right)  \right\vert , \text{ } N =
\sup_{\left(  \lambda,t\right)  \in I_{\delta}\times\Lambda} \left\vert
\left.  \frac{\partial\varphi}{\partial\lambda}\right\vert _{\left(
\lambda,t\right)  } \right\vert
\]
When $n>m-1$, we have:
\[
\left\vert \varphi_{n/m}\left(  \lambda,t\right)  \right\vert = \left\vert
\frac{\epsilon\int\varphi J^{n}}{J^{m}}\right\vert \leq\frac{M\int J^{n}%
}{J^{m}} \underset{\left(  \lambda,t\right)  \rightarrow\left(  \lambda
_{0},0\right)  }{\rightarrow} 0
\]
If $n=m-1$, then using the mean value theorem, for each $\left(
\lambda,t\right)  \in I_{\delta}\times\Lambda_{\delta}$:
\[
\varphi\left(  \lambda,t\right)  - \varphi\left(  \lambda_{0},t\right)  =
\left(  \lambda- \lambda_{0}\right)  \left.  \frac{\partial\varphi}%
{\partial\lambda}\right\vert _{\left(  \lambda_{1},t\right)  } \text{ for some
} \left(  \lambda_{1},t\right)  \in I_{\delta}\times\Lambda_{\delta}
\]
therefore:
\[
\left\vert \varphi\left(  \lambda,t\right)  - \varphi\left(  \lambda
_{0},t\right)  \right\vert \leq\left\vert \lambda- \lambda_{0}\right\vert N
\]
and thus:
\begin{align*}
&  \left\vert \varphi_{n/m}\left(  \lambda,t\right)  - \frac{1}{m}
\varphi\left(  \lambda,0\right)  \right\vert \\
&  \leq\left\vert \varphi_{n/m}\left(  \lambda,t\right)  - \varphi
_{n/m}\left(  \lambda_{0},t\right)  \right\vert + \left\vert \varphi
_{n/m}\left(  \lambda_{0},t\right)  - \frac{1}{m}\varphi\left(  \lambda
_{0},0\right)  \right\vert \\
&  = \left\vert \int\left(  \varphi\left(  \lambda,t\right)  - \varphi\left(
\lambda_{0},t\right)  \right)  J^{m-1}\right\vert + \left\vert \varphi
_{n/m}\left(  \lambda_{0},t\right)  - \frac{1}{m}\varphi\left(  \lambda
_{0},0\right)  \right\vert \\
&  \leq\int\left\vert \varphi\left(  \lambda,t\right)  - \varphi\left(
\lambda_{0},t\right)  \right\vert + \left\vert \varphi_{n/m}\left(
\lambda_{0},t\right)  - \frac{1}{m}\varphi\left(  \lambda_{0},0\right)
\right\vert \underset{\left(  \lambda,t\right)  \rightarrow\left(  \lambda
_{0},0\right)  }{\rightarrow} 0
\end{align*}

\end{proof}

\paragraph{Rule P.}

Using integration by parts:
\begin{equation}
\left.
\begin{array}
[c]{c}%
\varphi\in\mathcal{C}^{1}\\
n+1>0
\end{array}
\right\}  \Rightarrow\int\varphi J^{n}=\frac{\epsilon}{n+1}\left\{  \varphi
J^{n+1}-\int\varphi^{\prime}J^{n+1}\right\}  \label{RP}%
\end{equation}

\subsection{Proof of Theorem \ref{Baldo gen}}

We can write $F$ in the form:
\[
F=\epsilon\left(  \epsilon\int\psi J^{r}\right)  ^{\beta},\text{ }\beta
=\frac{1}{r+1}%
\]

\[%
\begin{tabular}
[c]{|l|}\hline
Since all variables $\left(  \lambda_{1},\ldots,\lambda_{n}\right)  \in
\Lambda$\\
play the same role, we will take $n=1$\\
that is, we take $\lambda= \lambda_{1}$.\\\hline
\end{tabular}
\]

We will freely use the notations introduced in Section \ref{Notacion}. Clearly
$F\in C^{0}$ since $\lim_{t\rightarrow0}F=\epsilon\lim_{t\rightarrow0}\left(
\epsilon\int\psi J^{r}\right)  ^{\beta}=0,$ and $F\left(  \lambda,0\right)
=0$.

To show that $F\in C^{1}$, we first prove that $F:t\mapsto F(\lambda,t)$
satisfies the hypotheses of Lemma \ref{lemain}. We have:
\begin{equation}
F^{\prime}=\left(  r+1\right)  ^{-\frac{1}{r+1}}\left(  \psi_{1}+F_{1}\right)
^{-\frac{r}{r+1}}\text{ with }\left\{  \text{
\begin{tabular}
[c]{l}%
$\psi_{1}=\psi^{-1/r}\in C^{\infty}$\\
$F_{1}=-\psi^{-\frac{r+1}{r}}J^{-r-1}\int\psi^{\prime}J^{\left(  r+1\right)
}$%
\end{tabular}
}\right.  \label{P1}%
\end{equation}

We note that $\lim F_{1}=0$ since, using rule H (\ref{RH}):$\lim J^{-r-1}%
\int\psi^{\prime}J^{\left(  r+1\right)  }=0$ and since $\psi_{1}>0$, we have:
\[
\lim F^{\prime}=\left(  r+1\right)  ^{-\frac{1}{r+1}}\psi_{1}\left(
\lambda,0\right)  ^{-\frac{1}{r}}%
\]
by Lemma \ref{lemain}, $F^{\prime}\left(  \lambda,t\right)  $ exists in
$\Lambda\times I_{\varepsilon}$ and $F^{\prime}\left(  \lambda,0\right)
=\left(  r+1\right)  ^{-\frac{1}{r+1}}\psi\left(  \lambda,0\right)  ^{\frac
{1}{r+1}}>0$

To proceed, it is essential to remark that $F^{\prime}\simeq F_{1}$ since by
(\ref{P1}), if $F_{1}\in\mathcal{C}^{k}$, then because $\psi_{1}\in C^{\infty
}$, it follows that $F^{\prime}\in C^{k}$; and if $F^{\prime}\in C^{k}$, then
$F_{1}\in C^{k}$ because $F_{1}=\left(  r+1\right)  ^{-\frac{1}{r}}%
F^{\prime-\frac{r+1}{r}}-\psi_{1}$

Therefore:
\begin{equation}%
\begin{tabular}
[c]{l}%
$F^{\prime}\simeq F_{1}=B_{1}\int\psi^{\prime}J^{\left(  r+1\right)  }\in
C^{\omega}$\\
$B_{1}=\psi_{11}J^{-r-1}$, $\psi_{11}=-\psi^{-\frac{r+1}{r}}\in C^{\infty}$%
\end{tabular}
\ \label{k=1}%
\end{equation}
and using rule H (\ref{RH}):$\lim F_{1}=\lim\left(  \psi_{11}J^{-r-1}\right)
=0$ and by rule H (\ref{RH2}), $F^{\prime}\simeq F_{1}\in C^{\omega}$.%

\begin{equation}%
\begin{tabular}
[c]{|l|}\hline
Notation: Henceforth, functions\\
denoted by $\psi_{i},$ $\psi_{ij},$ $\varphi_{i},$ $\varphi_{ij}$ are in
$C^{\infty}$.\\\hline
\end{tabular}
\label{RPsi}%
\end{equation}

Furthermore:
\[
\frac{\partial F_{1}}{\partial\lambda}=\frac{\partial\psi_{11}}{\partial
\lambda}J^{-r-1}\int\psi^{\prime}J^{\left(  r+1\right)  }+\psi_{11}%
J^{-r-1}\int\frac{\partial\psi^{\prime}}{\partial\lambda}J^{r+1}%
\]
It can then be proved using an inductive argument that for all $q\geq1$:
$\frac{\partial^{q}F_{1}}{\partial\lambda^{q}}=\sum\psi_{i}J^{-r-1}\int
\varphi_{i}J^{r+1}$ and then by the generalized rule H (\ref{RH1}),
$\frac{\partial F_{1}}{\partial\lambda}\in C^{\omega}$. By Lemma \ref{lemain},
it is concluded that:$\exists\left.  \frac{\partial^{q}F_{1}}{\partial
\lambda^{q}}\right\vert _{\left(  \lambda,t\right)  }\forall\left(
\lambda,t\right)  ,$ and $\frac{\partial^{q}F_{1}}{\partial\lambda^{q}}\in
C^{\omega}$ thus, since $F^{\prime}\simeq F_{1}$, we have proved that $F\in
C^{1}$ and $\exists\left.  \frac{\partial^{q}F^{\prime}}{\partial\lambda^{q}%
}\right\vert _{\left(  \lambda,t\right)  }\forall\left(  \lambda,t\right)  $
and $\forall q$

Let's see that $F\in C^{2}$. Indeed:$F_{1}^{\prime}=B_{1}^{\prime}\int
\psi^{\prime}J^{\left(  r+1\right)  }+B_{1}\psi^{\prime}J^{r+1}$ but the
second term is $C^{\infty}$ since $B_{1}\psi^{\prime}J^{\left(  r+1\right)
}=\psi_{11}\psi^{\prime}$, so it follows:$F_{1}^{\prime}\cong B_{1}^{\prime
}\int\psi^{\prime}J^{r+1}$ thus, taking $\psi_{21}=\psi_{11}^{\prime}$ and
$\psi_{22}=-r-1/\psi_{11}$ we get:$B_{1}^{\prime}=\psi_{21}J^{-r-1}%
+\epsilon\psi_{22}J^{-r-2}$ and by rule (\ref{RP}) $\int\psi^{\prime}%
J^{r+1}=\frac{\epsilon}{r+2}\left(  \psi^{\prime}J^{r+2}-\int\psi
^{\prime\prime}J^{r+2}\right)  $ we have:
\begin{align*}
F_{1}^{\prime}  &  \cong\frac{1}{r+2}\left\{  \epsilon B_{1}^{\prime}%
\psi^{\prime}J^{r+2}-\epsilon B_{1}^{\prime}\int\psi^{\prime\prime}%
J^{r+2}\right\} \\
&  \cong\epsilon B_{1}^{\prime}\int\psi^{\prime\prime}J^{r+2}%
\end{align*}
since: $\epsilon B_{1}^{\prime}\psi^{\prime}J^{\left(  r+2\right)  }=\left(
\epsilon\psi_{11}J^{-r-1}+\psi_{12}J^{-r-2}\right)  J^{r+2}=\psi_{11}\left(
\epsilon J\right)  +\psi_{12}=\psi_{11}+\psi_{12}\in C^{\infty}$

We thus have:
\begin{align}
F_{1}^{\prime}  &  \cong F_{2}=B_{2}\int\psi^{\prime\prime}J^{r+2}%
\label{k=2}\\
B_{2}  &  =\epsilon B_{1}^{\prime}=\epsilon\psi_{21}J^{-r-1}+\psi_{22}%
J^{-r-2},\text{ }\psi_{ij}\in C^{\infty}\nonumber
\end{align}

But now it can be proved, using the rule (\ref{RH1}) (\ref{RH2}) and Lemma
\ref{lemain} repeatedly, that $F_{1}^{\prime}\in C^{\omega}$, and for each
$q=1,2..,\exists\frac{\partial^{q}F_{2}}{\partial\lambda^{q}}\in C^{\omega}$

And since $F^{\prime}\simeq F_{1}\cong F_{2}$ and $F_{1}^{\prime}\in
C^{\omega}$, it is concluded that $F^{\prime\prime}\in C^{\omega}$, and for
$k=1,2...$, $\exists\frac{\partial^{k}F^{\prime}}{\partial\lambda^{k}}$ Thus
$F\in C^{2}.$

We will prove the following Lemma:

\begin{lemma}
\label{L1} Let $B_{1}=\psi_{11}J^{-r-1}$, $\psi_{11}=-\psi^{-\frac{r+1}{r}}$,
and $F_{0}=F$. For each $k\geq1$, the functions:
\begin{equation}
B_{k+1}=\epsilon B_{k}^{\prime},\text{ }F_{k}=B_{k}\int\psi^{\left(  k\right)
}J^{r+k} \label{bkp}%
\end{equation}
satisfy:
\begin{equation}
F_{k-1}^{\prime}\cong F_{k}\in C^{\omega} \label{fkp}%
\end{equation}%
\[
\exists\left.  \frac{\partial^{q}F_{k}}{\partial\lambda^{q}}\right\vert
_{\left(  \lambda,t\right)  }\forall\left(  \lambda,t\right)  ,\text{ and
}\frac{\partial^{q}F_{k}}{\partial\lambda^{q}}\in C^{\omega},\forall q
\]

\end{lemma}

\subsubsection{End of the Proof of Theorem \ref{Baldo gen}}

Assuming Lemma \ref{L1} is proven and knowing that each $F_{k}$ admits all
partial derivatives with respect to $\lambda$, the following implications are
valid:
\[%
\begin{tabular}
[c]{l}%
$F^{\prime}\simeq F_{1}$\\
$F_{1}^{\prime}\cong F_{2}\in C^{\omega}\Rightarrow F^{\prime}\in C^{1}$\\
$F_{2}^{\prime}\cong F_{3}\in C^{\omega}\Rightarrow F^{\prime}\in C^{2}$\\
etc...
\end{tabular}
\
\]
and we arrive at:
\begin{equation}
\exists\frac{\partial^{q+s}F}{\partial\lambda^{q}\partial t^{s}}\in C^{\omega
}\text{ if }s\geq1 \label{Drds}%
\end{equation}
but the problem now is to prove that:
\[
\exists\frac{\partial^{q}F}{\partial\lambda^{q}}\in C^{\omega}%
\]

Indeed, substituting in Lemma \ref{lemain} the function $f\left(  t\right)  $
by the function $f^{\lambda}\left(  t\right)  =F\left(  t,\lambda\right)  $
for a fixed $\lambda\in\Lambda$, and taking into account that $F\left(
0,\lambda\right)  =0,\forall\lambda$, we have:
\[
F\left(  t,\lambda\right)  =t\int_{0}^{1}\frac{\partial F}{\partial t}\left(
st,\lambda\right)  ds
\]
using now (\ref{Drds}), it is seen that:
\[
\frac{\partial^{q}F}{\partial\lambda^{q}}=t\int_{0}^{1}\frac{\partial
F}{\partial\lambda^{q}\partial t}\left(  st,\lambda\right)
\]
which belongs to $C^{\omega}$.

\subsubsection{Proof of Lemma \ref{L1}}

\begin{proof}
Note first that the $B_{k}$ for $k\geq1$are of the form (for certain
$\psi_{ij}\in C^{\infty}$):
\[
B_{k}=\epsilon B_{k-1}^{\prime}=\sum_{\ell=1}^{k}\epsilon^{k-\ell}\psi
_{k,\ell}J^{-r-\ell}%
\]

Observe first that the $F_{k}=B_{k}\int\psi^{\left(  k\right)  }J^{r+k}$ thus
constructed satisfy $F_{k}\in C^{\omega}$ since using the rule H (\ref{RH2})
we have:
\[
\epsilon^{k-\ell}\psi_{k,\ell}J^{-r-\ell}\int\psi^{\left(  k\right)  }%
J^{r+k}\in C^{\omega}\text{ for }1\leq\ell\leq k
\]
furthermore, as $\left(  r+\ell\right)  <\left(  r+k\right)  +1$ for
$1\leq\ell\leq k$, by (\ref{RH2}):
\[
\lim F_{k}=\lim\sum_{\ell=1}^{k}\epsilon^{k-\ell}\psi_{k,\ell}J^{-r-\ell}%
\int\psi^{\left(  k\right)  }J^{r+k}=0
\]

For $k=1$, the lemma was already proven in (\ref{k=1}).

The proof of the lemma for $k\geq2$ is done by induction on $k$. We have
proven the Lemma for $k=2$ in (\ref{k=2}), i.e., $F_{1}^{\prime}\cong F_{2}$.

Assuming the induction hypothesis (\ref{fkp}) $F_{k-1}^{\prime}\cong F_{k}\in
C^{\omega}$ and (\ref{bkp}) $F_{k}=B_{k}\int\psi^{\left(  k\right)  }J^{r+k}$
for some $k>2$, we have:
\[
F_{k}^{\prime}=B_{k}^{\prime}\int\psi^{\left(  k\right)  }J^{r+k}+B_{k}%
\psi^{\left(  k\right)  }J^{r+k}%
\]
but the second term is $C^{\infty}$ since:
\begin{align*}
B_{k}\psi^{\left(  k\right)  }J^{r+k}  &  =\left(  \sum_{\ell=1}^{k}%
\epsilon^{k-\ell}\psi_{k,\ell}J^{-r-\ell}\right)  \psi^{\left(  k\right)
}J^{r+k}\\
&  =\sum_{\ell=1}^{k}\psi_{k,\ell}\psi^{\left(  k\right)  }\left(
\epsilon^{k-\ell}J^{k-\ell}\right)  \in C^{\infty}\text{ since }\epsilon J=I
\end{align*}
so $F_{k}=B_{k}\int\psi^{\left(  k\right)  }J^{\left(  r+k\right)  }$, and
$F_{k}^{\prime}\cong B_{k}^{\prime}\int\psi^{\left(  k\right)  }J^{\frac
{2k+1}{2}}$ applying now rule P (\ref{RP}), we have:
\[
\int\psi^{\left(  k\right)  }J^{k+r}=\frac{\epsilon}{k+r+1}\left\{
\psi^{\left(  k\right)  }J^{k+r+1}-\int\psi^{\left(  k+1\right)  }%
J^{k+r+1}\right\}
\]
therefore:
\begin{align*}
F_{k}^{\prime}  &  =\frac{1}{k+r+1}\left(  \epsilon B_{k}^{\prime}\right)
\left\{  \psi^{\left(  k\right)  }J^{k+r+1}-\int\psi^{\left(  k+1\right)
}J^{k+r+1}\right\} \\
&  \cong B_{k+1}\int\psi^{\left(  k+1\right)  }J^{k+r+1}%
\end{align*}
since $\epsilon B_{k}^{\prime}\psi^{\left(  k\right)  }J^{k+r+1}\in C^{\infty
}$, because:
\begin{align*}
\epsilon B_{k}^{\prime}\psi^{\left(  k\right)  }J^{k+\ell+1}  &  =B_{k+1}%
\psi^{\left(  k\right)  }J^{k+r+1}\\
&  =\left(  \sum_{\ell=1}^{k+1}\epsilon^{k-\ell+1}\psi_{k+1,\ell}J^{-r-\ell
}\right)  \psi^{\left(  k\right)  }J^{k+r+1}\\
&  =\sum_{\ell=1}^{k+1}\psi_{k+1,\ell}\psi^{\left(  k\right)  }\left(
\epsilon J\right)  ^{k-\ell+1}\in C^{\infty}%
\end{align*}
where we have used $\epsilon J=I$.

Thus the conditions (\ref{fkp}) and (\ref{bkp}) of the induction hypothesis
hold for $k+1$:
\[
F_{k}^{\prime}\cong F_{k+1}=B_{k+1}\int\psi^{\left(  k+1\right)  }J^{k+r+1}\in
C^{\omega}\text{ with }B_{k+1}=\epsilon B_{k}^{\prime}%
\]

\end{proof}

\end{document}